\documentclass[reqno,12pt]{amsart}
\usepackage[margin=1.1in, footskip=1cm]{geometry}
\usepackage{amsmath, amsthm, amssymb, booktabs}
\usepackage{mathrsfs}
\usepackage[usenames,dvipsnames]{color}
\pdfpagewidth=\paperwidth
\pdfpageheight=\paperheight

\newtheorem{theorem}{Theorem}[section]
\newtheorem{lemma}[theorem]{Lemma}
\newtheorem{corollary}[theorem]{Corollary}

\theoremstyle{definition}

\theoremstyle{remark}

\numberwithin{equation}{section}

\def\le{\leqslant} \def\ge{\geqslant}

\makeatletter
\@namedef{subjclassname@2020}{\textup{2020} Mathematics Subject Classification}
\makeatother

\begin{document}
\title[Structure and paucity]{Structure and paucity in affine diagonal systems, I}
\author[J. Brandes]{Julia Brandes}
\address{Mathematical Sciences, University of Gothenburg and Chalmers Institute of Technology, 412 96 
G\"oteborg, Sweden}
\email{brjulia@chalmers.se}
\author[T. D. Wooley]{Trevor D. Wooley}
\address{Department of Mathematics, Purdue University, 150 N. University Street, West 
Lafayette, IN 47907-2067, USA}
\email{twooley@purdue.edu}
\subjclass[2020]{11D45, 11D72, 11D41}
\keywords{Paucity, diagonal equations, Diophantine equations in many variables}
\date{}

\begin{abstract} Let $\varepsilon>0$ and $\mathbf h\in \mathbb Z^3$. We show that whenever $P$ is large and 
the system
\[
x_1^j+x_2^j-y_1^j-y_2^j=h_j\quad (j=1,2,3)
\]
has more than $P^\varepsilon$ integral solutions with $1\le x_i,y_i\le P$, then there exist natural numbers $a$ and 
$b$ with $h_j=a^j-b^j$ $(j=1,2,3)$. This example illustrates the theme that, either the Diophantine system has a 
paucity of integral solutions, or else the coefficient tuple $\mathbf h$ is highly structured. We examine related 
paucity problems as well as some consequences for problems involving more variables.
\end{abstract}

\maketitle

\section{Introduction} This paper concerns the influence of coefficient structure on the paucity of integral solutions 
in certain systems of affine diagonal equations. What we have in mind is best illustrated by introducing our first 
example of interest. Let $P$ be a large natural number, and consider a coefficient triple 
$(h_1,h_2,h_3)\in \mathbb Z^3$. We denote by $S_2(P;\mathbf h)$ the number of integral solutions of the system
\begin{equation}\label{1.1}
\left.
\begin{aligned}
x_1^3+x_2^3-y_1^3-y_2^3&=h_3\\
x_1^2+x_2^2-y_1^2-y_2^2&=h_2\\
x_1+x_2-y_1-y_2&=h_1
\end{aligned}
\right\} ,
\end{equation}
with $1\le x_i,y_i\le P$ $(i=1,2)$. It is apparent that for certain triples $\mathbf h=(h_1,h_2,h_3)$, there are many 
integral solutions to the system \eqref{1.1} counted by $S_2(P;\mathbf h)$. When $\mathbf h=(0,0,0)$, for example, 
there are $2P^2-P$ diagonal solutions with $\{x_1,x_2\}=\{y_1,y_2\}$. Furthermore, when $a$ and $b$ are distinct 
natural numbers with $1\le a,b\le P$ and
\begin{equation}\label{1.2}
h_j=a^j-b^j\quad (j=1,2,3),
\end{equation}
there are $P$ solutions with $x_1=y_1$, $x_2=a$, $y_2=b$. Our goal is to prove that there is a paucity of integral 
solutions not belonging to one of these structured sets.

\begin{theorem}\label{theorem1.1}
Let $\eta\in (0,1)$ be fixed, and let $P$ be sufficiently large in terms of $\eta$. Suppose that 
$\mathbf h\in \mathbb Z^3$ is a coefficient triple with the property that $S_2(P;\mathbf h)>P^\eta$. Then, the 
tuple $\mathbf h$ satisfies one of the following two conditions:
\begin{itemize}
\item[(a)] for some integers $a$ and $b$ with $a\ne b$ and $1\le a,b\le P$, one has $h_j=a^j-b^j$ $(j=1,2,3)$, in 
which case $S_2(P;\mathbf h)=4P$;
\item[(b)] one has $h_j=0$ $(j=1,2,3)$, in which case $S_2(P;\mathbf 0)=2P^2-P$.
\end{itemize}
In particular, if $\mathbf h$ satisfies neither condition (a) nor condition (b), then 
$S_2(P;\mathbf h)=O(P^\varepsilon)$.
\end{theorem}

We note that in scenario (a), the solutions of \eqref{1.1} take the form $\mathbf x,\mathbf y$ with
\[
\{x_1,x_2\}=\{a,z\}\quad \text{and}\quad \{y_1,y_2\}=\{b,z\},
\]
wherein $z$ is any integer with $1\le z\le P$. Solutions in scenario (b), meanwhile, are the diagonal ones with 
$\{x_1,x_2\}=\{y_1,y_2\}$.\par

The wider theme that we explore in this paper is illustrated well by Theorem \ref{theorem1.1}. Either the 
system \eqref{1.1} has very few integral solutions, or else the coefficient tuple $\mathbf h$ is highly structured, with 
an explicit structure aligned with the equations comprising \eqref{1.1}. These ideas extend to systems in a larger 
number of variables. When $t\in \mathbb N$, denote by $S_t(P;\mathbf h)$ the number of integral solutions of the 
system
\[
\sum_{i=1}^t(x_i^j-y_i^j)=h_j\quad (j=1,2,3),
\]
with $1\le x_i,y_i\le P$ $(1\le i\le t)$. Improving on earlier work of Brandes and Hughes \cite{BH2022}, it was shown 
in Wooley \cite[Theorem 1.1]{Woo2023b} that $S_t(P;\mathbf h)\ll P^{t-1/2+\varepsilon}$ whenever $h_1\ne 0$ and 
$1\le t\le 5$, and also when $h_2\ne 0$ and $1\le t\le 4$. We note also that an asymptotic fomula for 
$S_6(P;\mathbf h)$ is established in \cite[Theorems 1.1 and 1.2]{Woo2023a}, provided that either $h_1\ne 0$ or 
$h_2\ne 0$. In addition, when $h_1=0$ and $h_2\ne 0$, it follows from \cite[Theorem 1.2]{Woo2023b} that one has 
$S_3(P;\mathbf h)\ll P^{2+\varepsilon}$. This, in common with all of the previous results cited, offers estimates 
beyond square-root cancellation, with the upper bound $S_3(P;\mathbf h)\ll P^{2+\varepsilon}$ saving a factor 
$P^{1-\varepsilon}$ over this square-root barrier when $h_1=0$ and $h_2\ne 0$. An immediate consequence of 
Theorem \ref{theorem1.1} is that this last upper bound now extends to all $\mathbf h\ne \mathbf 0$.

\begin{corollary}\label{corollary1.2}
Suppose that $\mathbf h\in \mathbb Z^3\setminus \{\mathbf 0\}$ and $\varepsilon>0$. Then one has 
$S_3(P;\mathbf h)\ll P^{2+\varepsilon}$.
\end{corollary}

We remark that when $\mathbf h$ has the shape determined by \eqref{1.2}, with $1\le a,b\le P$ and 
$a\ne b$, then $S_3(P;\mathbf h)\gg P^2$. This conclusion follows via an argument that the reader will readily 
discern from the proof of the corollary. Thus, the conclusion of the corollary may be regarded as essentially sharp.

\par It will be evident from our discussion of the proof of Theorem \ref{theorem1.1} in \S2 that the system 
\eqref{1.1} has particularly convenient features associated with its status as an affine Vinogradov system. In other 
systems, additional difficulties may be encountered. We illustrate these challenges with a consideration of a system 
closely related to \eqref{1.1}. Consider then a coefficient triple $(h_1,h_2,h_4)\in \mathbb Z^3$. We denote by 
$T_2(P;\mathbf h)$ the number of integral solutions of the system
\begin{equation}\label{1.3}
\left.
\begin{aligned}
x_1^4+x_2^4-y_1^4-y_2^4&=h_4\\
x_1^2+x_2^2-y_1^2-y_2^2&=h_2\\
x_1+x_2-y_1-y_2&=h_1
\end{aligned}
\right\} ,
\end{equation}
with $1\le x_i,y_i\le P$ $(i=1,2)$. In \S3 we establish the following relative of Theorem \ref{theorem1.1}.

\begin{theorem}\label{theorem1.3} Let $\eta\in (0,1)$ be fixed, and let $P$ be sufficiently large in terms of $\eta$. 
Suppose that $\mathbf h\in \mathbb Z^3$ is a coefficient triple with the property that $T_2(P;\mathbf h)>P^\eta$. 
Then, the tuple $\mathbf h$ satisfies one of the following two conditions:
\begin{itemize}
\item[(a)] for some integers $a$ and $b$ with $a\ne b$ and $1\le a,b\le P$, one has $h_j=a^j-b^j$ $(j=1,2,4)$, in 
which case $T_2(P;\mathbf h)=4P$;
\item[(b)] one has $h_j=0$ $(j=1,2,4)$, in which case $T_2(P;\mathbf 0)=2P^2-P$.
\end{itemize}
In particular, if $\mathbf h$ satifies neither condition (a) nor condition (b), then 
$T_2(P;\mathbf h)=O(P^\varepsilon)$.
\end{theorem}

Similar comments apply concerning the structure of solutions in scenarios (a) and (b) as discussed following the 
statement of Theorem \ref{theorem1.1}. Moreover, just as in the analogous discussion relating to 
$S_t(P;\mathbf h)$, one can examine the situation with additional variables. Denote by $T_t(P;\mathbf h)$ the 
number of integral solutions of the system
\begin{equation}\label{1.4}
\sum_{i=1}^t(x_i^j-y_i^j)=h_j\quad (j=1,2,4),
\end{equation}
with $1\le x_i,y_i\le P$ $(1\le i\le t)$. We present an analogue of Corollary \ref{corollary1.2} relating to the system 
\eqref{1.4}.

\begin{corollary}\label{corollary1.4}
Suppose that $\mathbf h\in \mathbb Z^3\setminus \{ \mathbf 0\}$ and $\varepsilon >0$. Then one has 
$T_3(P;\mathbf h)\ll P^{2+\varepsilon}$.
\end{corollary}

Just as in our remark following the statement of Corollary \ref{corollary1.2}, this conclusion is essentially sharp in 
the case that $h_j=a^j-b^j$ $(j=1,2,4)$, with $1\le a,b\le P$ and $a\ne b$.\par

We complete this initial foray into the interplay between structure and paucity with another illustrative example. 
When $s$ and $k$ are natural numbers, and $\mathbf h\in \mathbb Z^k$, denote by $U_{s,k}(P;\mathbf h)$ the 
number of solutions of the Diophantine system
\begin{equation}\label{1.5}
\sum_{i=1}^sx_i^{2j-1}=h_j\quad (1\le j\le k),
\end{equation}
with $|x_i|\le P$ $(1\le i\le s)$. This system is an affine version of that considered in work of Br\"udern and Robert 
\cite{BR2012}. We first offer a theorem analogous to Theorems \ref{theorem1.1} and \ref{theorem1.3} showing that 
limitations to the strength of paucity results imply structure in the tuple $\mathbf h$.

\begin{theorem}\label{theorem1.5}
Let $\eta\in (0,1)$ be fixed, and let $k$ and $r$ be non-negative integers with $k\ge 2$. Let $P$ be sufficiently large 
in terms of $\eta$, $k$ and $r$. Suppose that $\mathbf h\in \mathbb Z^k$ is a coefficient $k$-tuple having the 
property that
\begin{equation}\label{1.6}
U_{k+1,k}(P;\mathbf h)> P^{r+\eta}.
\end{equation}
Then one has $0\le r\le (k-1)/2$, and for some integers $a_i$ with $|a_i|\le P$ $(1\le i\le k-1-2r)$, one has
\begin{equation}\label{1.7}
h_j=a_1^{2j-1}+\ldots +a_{k-1-2r}^{2j-1}\quad (1\le j\le k).
\end{equation}
In particular, when $r$ is the largest integer for which the lower bound \eqref{1.6} holds for some $\eta\in (0,1)$, 
then for every $\varepsilon>0$, we have
\[
P^{r+1}\ll U_{k+1,k}(P;\mathbf h)\ll P^{r+1+\varepsilon}.
\]
\end{theorem}

We present a conclusion essentially equivalent to Theorem \ref{theorem1.5} having a converse flavour, 
demonstrating that structure in the coefficient tuple $\mathbf h$ implies precise conclusions concerning paucity 
estimates.

\begin{theorem}\label{theorem1.6}
Let $P$ be large enough in terms of $k$, and suppose that $\mathbf h \in \mathbb Z^k \setminus \{\mathbf 0\}$ is 
a coefficient $k$-tuple. Take $r_0(\mathbf h)$ to be the smallest natural number $r$ with the property that there 
exist integers $a_i$ with $|a_i|\le P$ $(1\le i\le r)$ satisfying the equations 
\begin{equation}\label{1.8}
h_j=a_1^{2j-1}+\ldots +a_r^{2j-1}\quad (1\le j\le k),
\end{equation}
and write
\[
\tau(k;\mathbf h)=\left\lfloor \frac{k+1-r_0(\mathbf h)}{2}\right\rfloor .
\]
Then for every $\varepsilon>0$, we have
\[
P^{\tau(k;\mathbf h)}\le U_{k+1,k}(P;\mathbf h)\ll P^{\tau(k;\mathbf h)+\varepsilon}.
\]
\end{theorem}

These theorems show that large values of the counting function $U_{k+1,k}(P;\mathbf h)$ are associated with a high degree of structure in the coefficient $k$-tuple $\mathbf h$. Here, the fewer the number of 
summands in \eqref{1.7} required to define $\mathbf h$, the greater the implied structural constraints on 
$\mathbf h$. Similar conclusions may be derived for the quantity $U_{s,k}(P;\mathbf h)$ when $s>k+1$, though we 
choose not to elaborate on this matter herein. Instead, we present a non-trivial paucity result that follows from 
Theorem \ref{theorem1.5} in a manner analogous to the corollaries presented above. We emphasise that this 
conclusion goes beyond the so-called square-root barrier, in the sense that the upper bound established for 
$U_{2t,k}(P;\mathbf h)$ grows more slowly than $P^t$, the square-root of the underlying available supply of 
variables.

\begin{corollary}\label{corollary1.7} Suppose that $k\ge 2$ and 
$\mathbf h\in \mathbb Z^k\setminus \{\mathbf 0\}$. Then, for each $\varepsilon>0$, one has
\[
U_{2t,k}(P;\mathbf h)\ll P^{t-1+\varepsilon}\quad (1\le t\le k).
\]
\end{corollary}

In common with previous work on paucity and quasi-paucity problems (see, for example \cite{PW2002, Woo1993}), 
our approach to proving Theorems \ref{theorem1.1}, \ref{theorem1.3}, \ref{theorem1.5} and \ref{theorem1.6} rests on the application 
of polynomial identities of multiplicative type. In fortunate circumstances, such identities permit divisor function 
estimates to be deployed showing that there are few integral solutions of the Diophantine systems of interest for 
typical tuples $\mathbf h$. For atypical choices of $\mathbf h$, however, one must contend with the possibility that 
these polynomial identities show only that the underlying variables are constrained in terms of divisors of $0$, 
thereby offering no useful information. One therefore requires a careful analysis of these atypical situations and 
their relation to diagonal structures within the associated Diophantine systems. The reader will gather from the 
discussion of \S2 the ideas necessary for the consideration of related problems.\par

Our basic parameter is $P$, a sufficiently large positive integer. Whenever $\varepsilon$ appears in a statement, 
either implicitly or explicitly, we assert that the statement holds for each $\varepsilon>0$. We make frequent use of 
vector notation in the form $\mathbf x=(x_1,\ldots ,x_r)$. Here, the dimension $r$ depends on the course of the 
argument.\par

Work on this paper was conducted while the first author was supported by Project Grant 2022-03717 from 
Vetenskapsr\aa det (Swedish Science Foundation), and while the second author was supported by NSF grant 
DMS-2502625 and Simons Fellowship in Mathematics SFM-00011955. The first author is grateful to Purdue 
University for its hospitality. The second author is grateful to the Institute 
for Advanced Study, Princeton, for hosting his sabbatical, during which period this paper was completed. Both 
authors thank Institut Mittag-Leffler and Mathematisches Forchungsinstitut Oberwolfach for hosting them during 
periods during which progress was made on this paper.

\section{Paucity in the affine cubic Vinogradov system}
Key to the proof of Theorem \ref{theorem1.1} is a polynomial identity having a multiplicative flavour. In order to 
describe this identity, we define the polynomials $s_j=s_j(x,y)$ by putting
\begin{equation}\label{2.1}
s_j(x,y)=x^j-y^j\quad (j\ge 1).
\end{equation}
One can check on the back of a small envelope that one has the identity
\begin{equation}\label{2.2}
s_1^4+3s_2^2-4s_1s_3=0.
\end{equation}
As a consequence of the relation \eqref{2.2}, one has an identity relating the $4$-variable polynomials 
$\sigma_j=\sigma_j(\mathbf x,\mathbf y)$, defined by
\begin{equation}\label{2.3}
\sigma_j(\mathbf x,\mathbf y)=x_1^j+x_2^j-y_1^j-y_2^j\quad (j\ge 1).
\end{equation}

\begin{lemma}\label{lemma2.1} One has
\begin{equation}\label{2.4}
\sigma_1(\mathbf x,\mathbf y)^4+3\sigma_2(\mathbf x,\mathbf y)^2-4\sigma_1(\mathbf x,\mathbf y)
\sigma_3(\mathbf x,\mathbf y)=12(x_1-y_1)(x_1-y_2)(x_2-y_1)(x_2-y_2).
\end{equation}
\end{lemma}

\begin{proof} If we specialise by setting $x_2=y_2$, then we find that
\[
\sigma_j(\mathbf x,\mathbf y)=s_j(x_1,y_1)\quad (j=1,2,3).
\]
Under such circumstances, the left hand side of \eqref{2.4} is equal to $s_1^4+3s_2^2-4s_1s_3$, which is $0$. It 
therefore follows that the left hand side of \eqref{2.4} is divisible by $x_2-y_2$, and symmetrical arguments reveal 
that $x_1-y_1$, $x_1-y_2$ and $x_2-y_1$ are also factors. The identity \eqref{2.4} then follows on observing that the left 
and right hand sides of \eqref{2.4} have the same degree $4$, with the constant factor $12$ being fixed by the 
coefficient of $x_1^2x_2^2$ on the left hand side of the equation.
\end{proof}

We are now equipped to confirm the conclusion of Theorem \ref{theorem1.1}.

\begin{proof}[Proof of Theorem \ref{theorem1.1}]
Let $\eta\in (0,1)$ and $\mathbf h\in \mathbb Z^3$. Whenever $\mathbf x,\mathbf y$ is an integral solution of 
\eqref{1.1} counted by $S_2(P;\mathbf h)$, it follows from \eqref{2.3} that $\sigma_j(\mathbf x,\mathbf y)=h_j$ 
$(j=1,2,3)$. Hence, from Lemma \ref{lemma2.1} we have
\begin{equation}\label{2.5}
12(x_1-y_1)(x_1-y_2)(x_2-y_1)(x_2-y_2)=\Psi (\mathbf h),
\end{equation}
in which
\[
\Psi(\mathbf h)=h_1^4+3h_2^2-4h_1h_3.
\]
Since $1\le x_i,y_i\le P$ $(i=1,2)$, we have $|\Psi(\mathbf h)|<12P^4$. We divide into cases.\par

Suppose first that $\Psi(\mathbf h)\ne 0$. Then a standard divisor function estimate shows that there are 
$O(P^\varepsilon)$ possible choices for integers $d_i$ $(1\le i\le 4)$ having the property that
\begin{equation}\label{2.6}
x_1-y_1=d_1,\quad x_1-y_2=d_2,\quad x_2-y_1=d_3,\quad x_2-y_2=d_4.
\end{equation}
Fixing any one such choice of $\mathbf d$, we may substitute
\begin{equation}\label{2.7}
y_1=x_1-d_1,\quad y_2=x_1-d_2,\quad x_2=x_1-d_1+d_3
\end{equation}
into the system \eqref{1.1}. Thus, we find that
\begin{equation}\label{2.8}
x_1^2+(x_1-d_1+d_3)^2-(x_1-d_1)^2-(x_1-d_2)^2=h_2,
\end{equation}
whence
\begin{equation}\label{2.9}
2x_1(d_2+d_3)+(d_1-d_3)^2-d_1^2-d_2^2=h_2.
\end{equation}
This equation fixes the value of $x_1$ unless $d_3=-d_2$, which we henceforth assume to be the case. Likewise, one 
has 
\[
x_1^3+(x_1-d_1+d_3)^3-(x_1-d_1)^3-(x_1-d_2)^3=h_3,
\]
whence
\[
3x_1((d_1+d_2)^2-d_1^2-d_2^2)-(d_1+d_2)^3+d_1^3+d_2^3=h_3.
\]
This equation fixes the value of $x_1$ unless
\begin{equation}\label{2.10}
(d_1+d_2)^2-d_1^2-d_2^2=0.
\end{equation}
We are therefore forced to conclude that the value of $x_1$ is fixed, unless one has $2d_1d_2=0$, which is to say 
that one at least of $d_1$ and $d_2$ is equal to $0$. The latter eventuality implies that $x_1=y_1$ or $x_1=y_2$. 
Either such circumstance implies, via \eqref{2.5}, that $\Psi(\mathbf h)=0$, contradicting our earlier assumption in 
this case. Thus $x_1$ is indeed determined via $\mathbf h$ and $\mathbf d$, so that we may infer from \eqref{2.7} 
that $y_1$, $y_2$ and $x_2$ are also determined. In this first scenario in which $\Psi(\mathbf h)\ne 0$, we therefore 
conclude that $S_2(P;\mathbf h)\ll P^\varepsilon$.\par

If it is the case that $S_2(P;\mathbf h)>P^\eta$, then whenever $P$ is large enough in terms of $\varepsilon$ and 
$\eta$, we cannot be in this first scenario. Thus, we have $\Psi(\mathbf h)=0$, and it follows from \eqref{2.5} that
\begin{equation}\label{2.11}
x_1=y_1,\quad \text{or}\quad x_1=y_2,\quad \text{or}\quad x_2=y_1,\quad \text{or}\quad x_2=y_2.
\end{equation}
By relabelling variables, there is no loss of generality in supposing that $x_2=y_2$. We then deduce from \eqref{1.1} 
that
\begin{equation}\label{2.12}
x_1^j-y_1^j=h_j\quad (j=1,2,3).
\end{equation}
In particular, there exist integers $a$ and $b$ with $1\le a,b\le P$ for which one has the relation \eqref{1.2}, namely 
$h_j=a^j-b^j$ $(j=1,2,3)$. If one has $a=b$, and hence $h_j=0$ $(j=1,2,3)$, then plainly $x_1=y_1$. Keeping in 
mind our relabelling of variables, we then have
\[
S_2(P;{\mathbf 0})=2P^2-P.
\]
When $a\ne b$, meanwhile, we find from \eqref{2.12} that
\[
x_1-y_1=a-b\quad \text{and}\quad x_1+y_1=\frac{x_1^2-y_1^2}{x_1-y_1}=\frac{a^2-b^2}{a-b}=a+b.
\]
Thus $x_1=a$ and $y_1=b$, and the solutions of \eqref{1.1} of this kind counted by $S_2(P;\mathbf h)$ satisfy
\[
\{x_1,x_2\}=\{a,t\}\quad \text{and}\quad \{y_1,y_2\}=\{b,t\},
\]
with $1\le t\le P$. Again noting our earlier relabelling of variables, we find that when $a\ne b$, one has
\[
S_2(P;\mathbf h)=4P.
\]
This completes the proof of the theorem.
\end{proof}

The corollary to Theorem \ref{theorem1.1} requires little additional effort to establish.

\begin{proof}[Proof of Corollary \ref{corollary1.2}] We suppose throughout that ${\mathbf h}\ne {\mathbf 0}$. One 
has
\begin{equation}\label{2.13}
S_3(P;\mathbf h)=\sum_{1\le x_3,y_3\le P}S_2(P;\mathbf h'(x_3,y_3)),
\end{equation}
where $\mathbf h'=(h_1',h_2',h_3')$ and
\begin{equation}\label{2.14}
h_j'(u,v)=h_j-u^j+v^j\quad (j\ge 1).
\end{equation}
Let $\eta>0$ be arbitrarily small. From Theorem \ref{theorem1.1}, we find that $S_2(P;\mathbf h')\le P^\eta$ unless, 
for some positive integers $a$ and $b$ with $1\le a,b\le P$, one has $h_j'=a^j-b^j$ $(j=1,2,3)$, in which case
\begin{equation}\label{2.15}
x_3^j-y_3^j+a^j-b^j=h_j\quad (j=1,2,3).
\end{equation}
When $a\ne b$, one then has $S_2(P;\mathbf h')=4P$, and when $a=b$ one has $S_2(P;\mathbf h')=2P^2-P$. We 
therefore conclude from \eqref{2.13} that
\begin{equation}\label{2.16}
S_3(P;\mathbf h)\ll P^{2+\eta}+P\Upsilon_1+P^2\Upsilon_2,
\end{equation}
where $\Upsilon_1$ denotes the number of solutions of \eqref{2.15} with $1\le a,b,x_3,y_3\le P$ and $a\ne b$, and 
$\Upsilon_2$ denotes the corresponding number of solutions with $a=b=1$.\par

A second application of Theorem \ref{theorem1.1} shows that $\Upsilon_1=O(P)$ whenever 
$\mathbf h\ne \mathbf 0$. Turning next to the task of estimating $\Upsilon_2$, in which case we may assume that 
$a=b=1$, we observe that the equations \eqref{2.15} yield $h_2=h_3=0$ whenever $h_1=0$. Consequently, when $\mathbf h \neq \mathbf 0$ we must have $h_1 \neq 0$, and then we find that 
\[
x_3-y_3=h_1\quad \text{and}\quad x_3+y_3=h_2/h_1.
\]
Thus, we infer that $x_3$ and $y_3$ are determined uniquely by $h_1$ and $h_2$ when $\mathbf h\ne {\mathbf 0}$. 
Hence, when $a=b=1$, one sees that $\Upsilon_2=O(1)$. Then we conclude from \eqref{2.16} that whenever 
$\mathbf h\ne \mathbf 0$, one has $S_3(P;\mathbf h)\ll P^{2+\varepsilon}$. This completes the proof of the 
corollary.
\end{proof}

We note that scenario (a) of Theorem \ref{theorem1.1} may be more explicitly described. Thus, if $h_j=a^j-b^j$ 
$(j=1,2,3)$ with $a\ne b$, then
\[
h_1=a-b\quad \text{and}\quad h_2/h_1=a+b,
\]
whence
\[
a=\tfrac{1}{2}(h_2/h_1+h_1)\quad \text{and}\quad b=\tfrac{1}{2}(h_2/h_1-h_1).
\]
One then has
\[
h_3=a^3-b^3=\frac{1}{8}\Bigl( \Bigl( \frac{h_2}{h_1}+h_1\Bigr)^3-\Bigl( \frac{h_2}{h_1}-h_1\Bigr)^3\Bigr),
\]
so that $4h_1h_3=h_1^4+3h_2^2$, as should have been apparent from the condition $\Psi(\mathbf h)=0$ already 
encountered in the proof of Theorem \ref{theorem1.1}.

\section{Paucity in a relative of the affine Vinogradov system}
In initiating the argument employed to establish Theorem \ref{theorem1.1}, one rapidly discovers that the system 
\eqref{1.3} central to Theorem \ref{theorem1.3} is more delicate to analyse than that underlying Theorem 
\ref{theorem1.1}. We again define the polynomials $s_j=s_j(x,y)$ as in \eqref{2.1}. Then, one may check on the back 
of a somewhat larger envelope that one has the identity
\begin{equation}\label{3.1}
s_2^3+s_1^4s_2-2s_1^2s_4=0.
\end{equation}
As a consequence of this relation, one has an identity relating the $4$-variable polynomials 
$\sigma_j(\mathbf x,\mathbf y)$ defined in \eqref{2.3}, together with the additional polynomial
\begin{equation}\label{3.2}
\tau(\mathbf x,\mathbf y)=(x_1^2+x_1x_2+x_2^2)-(y_1^2+y_1y_2+y_2^2).
\end{equation}

\begin{lemma}\label{lemma3.1}
One has
\begin{equation}\label{3.3}
\begin{aligned}\sigma_2(\mathbf x,\mathbf y)^3+&\sigma_1(\mathbf x,\mathbf y)^4
\sigma_2(\mathbf x,\mathbf y)-2\sigma_1(\mathbf x,\mathbf y)^2\sigma_4(\mathbf x,\mathbf y)\\
&=8(x_1-y_1)(x_1-y_2)(x_2-y_1)(x_2-y_2)\tau(\mathbf x,\mathbf y).\end{aligned}
\end{equation}
\end{lemma}

\begin{proof}
We again observe that by specialising $x_2=y_2$, one has
\[
\sigma_j(\mathbf x,\mathbf y)=s_j(x_1,y_1)\quad (j=1,2,4).
\]
Thus, in view of the identity \eqref{3.1}, it is evident as in the proof of Lemma \ref{lemma2.1} that the left hand side 
of \eqref{3.3} is divisible by $x_2-y_2$, with symmetrical arguments uncovering the additional factors $x_1-y_1$, 
$x_1-y_2$, and $x_2-y_1$. One may then divide the left hand side of \eqref{3.3} by the product
\[
(x_1-y_1)(x_1-y_2)(x_2-y_1)(x_2-y_2)
\]
to reveal the quotient $8\tau(\mathbf x,\mathbf y)$, with $\tau(\mathbf x,\mathbf y)$ defined by \eqref{3.2}. 
Readers wishing to conserve energy might choose to employ a computer algebra assistant in this computation. This completes the 
proof of the lemma.
\end{proof}

Although we are equipped to establish Theorem \ref{theorem1.3} using Lemma \ref{lemma3.1}, the details are more 
complicated than was the case in the corresponding argument employed in the proof of Theorem \ref{theorem1.1}.

\begin{proof}[Proof of Theorem \ref{theorem1.3}] Let $\eta\in (0,1)$ and $\mathbf h\in \mathbb Z^3$. Whenever 
$\mathbf x,\mathbf y$ is an integral solution of \eqref{1.3} counted by $T_2(P;\mathbf h)$, it follows from 
\eqref{2.3} that $\sigma_j(\mathbf x,\mathbf y)=h_j$ $(j=1,2,4)$. Hence, on applying Lemma \ref{lemma3.1}, we 
perceive that
\begin{equation}\label{3.4}
8(x_1-y_1)(x_1-y_2)(x_2-y_1)(x_2-y_2)\tau(\mathbf x,\mathbf y)=\Phi(\mathbf h),
\end{equation}
in which $\Phi(\mathbf h)=h_2^3+h_1^4h_2-2h_1^2h_4$. On this occasion, since $1\le x_i,y_i\le P$ $(i=1,2)$, we 
have $|\Phi(\mathbf h)|<24P^6$. We divide into cases.\par

Suppose first that $\Phi(\mathbf h)\ne 0$. Then a standard divisor function estimate shows that there are 
$O(P^\varepsilon)$ possible choices for integers $d_i$ $(1\le i\le 4)$ having the property that the relations 
\eqref{2.6} hold. Fixing any one such choice for $\mathbf d$, we may again substitute \eqref{2.7} into the system 
\eqref{1.3}. Just as in the proof of Theorem \ref{theorem1.1}, we again obtain \eqref{2.8} and \eqref{2.9}. The latter 
equation fixes the value of $x_1$ unless $d_3=-d_2$, which we henceforth assume to be the case. On this occasion, 
we have
\[
x_1^4+(x_1-d_1+d_3)^4-(x_1-d_1)^4-(x_1-d_2)^4=h_4,
\]
whence
\[
6x_1^2\bigl( (d_1+d_2)^2-d_1^2-d_2^2\bigr) -4x_1\bigl( (d_1+d_2)^3-d_1^3-d_2^3\bigr) 
+(d_1+d_2)^4-d_1^4-d_2^4=h_4.
\]
This equation fixes the value of $x_1$ to be one of the two roots of a quadratic equation unless \eqref{2.10} holds. 
We are therefore forced to conclude that the value of $x_1$ is fixed, unless one at least of $d_1$ and $d_2$ is $0$, 
an eventuality that again implies that $x_1=y_1$ or $x_1=y_2$. In either circumstance, it follows from \eqref{3.4} 
that $\Phi(\mathbf h)=0$, contradicting our earlier assumption. We thus conclude that $x_1$ is indeed determined 
by fixed choices for $\mathbf h$ and $\mathbf d$, whence from \eqref{2.7} we see that $y_1$, $y_2$ and $x_2$ are 
also determined. In this first scenario, in which $\Phi(\mathbf h)\ne 0$, therefore, we are forced to conclude that
\[
T_2(P;\mathbf h)\ll P^\varepsilon.
\]

\par If it is the case that $T_2(P;\mathbf h)>P^\eta$, then whenever $P$ is large enough in terms of $\varepsilon$ 
and $\eta$, we cannot be in the first scenario. Thus, we have $\Phi(\mathbf h)=0$, and it follows from \eqref{3.4} 
that, either one at least of the relations \eqref{2.11} must hold, or else
\begin{equation}\label{3.5}
\tau(\mathbf x,\mathbf y)=(x_1^2+x_1x_2+x_2^2)-(y_1^2+y_1y_2+y_2^2)=0.
\end{equation}

\par Before embarking on the analysis of this second case, we begin by examining the consequences of the assumption that the relation 
\eqref{3.5} holds. In this situation, it follows from the system of equations \eqref{1.3} that
\begin{align*}
(y_1+y_2+h_1)^2-&2(y_1^2+y_1y_2+y_2^2)+(y_1^2+y_2^2+h_2)\\
&=(x_1+x_2)^2-2(x_1^2+x_1x_2+x_2^2)+(x_1^2+x_2^2)=0,
\end{align*}
whence
\begin{equation}\label{3.6}
2h_1(y_1+y_2)+h_1^2+h_2=0.
\end{equation}
On substituting this relation into the linear equation of \eqref{1.3}, it follows that one also has
\begin{equation}\label{3.7}
2h_1(x_1+x_2)-h_1^2+h_2=0.
\end{equation}
We now make use of \eqref{3.6} and \eqref{3.7}, respectively, to deduce the two relations
\begin{align*}
8h_1^2(y_1^2+y_2^2+h_2)&=2\bigl( 4h_1^2(y_1^2+h_2)+(h_1^2+h_2+2h_1y_1)^2\bigr) \\
&=\bigl( 4h_1y_1+(h_1^2+h_2)\bigr)^2+8h_1^2h_2+(h_1^2+h_2)^2
\end{align*}
and
\begin{align*}
8h_1^2(x_1^2+x_2^2)&=2\bigl( 4h_1^2x_1^2+(h_1^2-h_2-2h_1x_1)^2\bigr) \\
&=\bigl( 4h_1x_1-(h_1^2-h_2)\bigr)^2+(h_1^2-h_2)^2.
\end{align*}
Hence, on substituting these formul{\ae} into the quadratic equation of \eqref{1.3}, we infer that
\begin{equation}\label{3.8}
\bigl( 4h_1x_1-(h_1^2-h_2)\bigr)^2-\bigl( 4h_1y_1+(h_1^2+h_2)\bigr)^2=12h_2h_1^2.
\end{equation}

\par We now begin our investigation of the scenario in which 
\[
\Phi(\mathbf h)=h_2^3+h_1^4h_2-2h_1^2h_4=0,
\]
but none of the relations \eqref{2.11} hold. In that scenario, we must therefore have \eqref{3.5}, and consequently 
\eqref{3.8}. Our goal is to show that these conditions are in fact incompatible with the assumption that 
$T_2(P;\mathbf h)>P^\eta$. In order to see this, consider first the case in which $h_1=0$. Then since 
$\Phi(\mathbf h)=0$, it follows that $h_2=0$, and we have
\[
\left.\begin{aligned}
x_1^2+x_2^2&=y_1^2+y_2^2\\
x_1+x_2&=y_1+y_2
\end{aligned}\right\}.
\]
This forces us to conclude that $\{ x_1,x_2\}=\{y_1,y_2\}$, violating our assumption that none of the relations 
\eqref{2.11} hold. It follows that we must have $h_1 \neq 0$.\par

Consider next the case where $h_1 \neq 0$ but $h_2=0$. This then implies that $h_4=0$, so that we must have
\[
\left.\begin{aligned}
x_1^4+x_2^4&=y_1^4+y_2^4\\
x_1^2+x_2^2&=y_1^2+y_2^2
\end{aligned}\right\} .
\]
These equations imply that $\{ x_1^2,x_2^2\}=\{y_1^2,y_2^2\}$, and since $1\le x_i,y_i\le P$, it follows that 
$\{x_1,x_2\}=\{y_1,y_2\}$. This again contradicts our assumption that none of the relations \eqref{2.11} hold.\par

We may therefore suppose that $12h_2h_1^2\ne 0$. Under such circumstances, an elementary divisor function estimate leads from \eqref{3.8} to the conclusion that there are at most 
$O(P^\varepsilon)$ possible choices for integers $e_1$ and $e_2$ for which
\[
\left. \begin{aligned}
\bigl( 4h_1x_1-(h_1^2-h_2)\bigr)-\bigl( 4h_1y_1+(h_1^2+h_2)\bigr) &=e_1\\
\bigl( 4h_1x_1-(h_1^2-h_2)\bigr)+\bigl( 4h_1y_1+(h_1^2+h_2)\bigr) &=e_2\
\end{aligned}
\right\}.
\]
Fixing any one choice for $e_1$ and $e_2$, it follows that both $x_1$ and $y_1$ are determined. We then find from 
\eqref{3.6} and \eqref{3.7} that $x_2$ and $y_2$ are also determined. Consequently, the number of solutions of this 
type is at most $O(P^\varepsilon)$. Since we had assumed that $T_2(P;\mathbf h) > P^{\eta}$, it follows that this scenario is also excluded. This affirms our claim that the assumption that none of the relations \eqref{2.11} hold is incompatible with the condition that $T_2(P; \mathbf h)>P^{\eta}$.  \par

It remains to consider the scenario in which one of the relations \eqref{2.11} does hold. Here, by relabelling variables, 
there is no loss of generality in supposing that $x_2=y_2$. We then deduce from \eqref{1.3} that
\begin{equation}\label{3.9}
x_1^j-y_1^j=h_j\quad (j=1,2,4).
\end{equation}
In particular, there exist integers $a$ and $b$ with $1\le a,b\le P$ for which one has $h_j=a^j-b^j$ $(j=1,2,4)$. If 
one has $a=b$, and hence $h_j=0$ $(j=1,2,4)$, then plainly $x_1=y_1$, and we are in case (b) of the theorem. In 
view of our relabelling of variables, one then has
\[
T_2(P;\mathbf 0)=2P^2-P.
\]
When $a\ne b$, meanwhile, we find from \eqref{3.9}, just as in the proof of Theorem \ref{theorem1.1}, that $x_1=a$ 
and $y_1=b$, and the solutions of the system \eqref{1.3} of this kind counted by $T_2(P;\mathbf h)$ satisfy
\[
\{x_1,x_2\}=\{a,t\}, \quad \{y_1,y_2\}=\{b,t\},
\]
with $1\le t\le P$. Again accounting for our earlier relabelling of variables, we see that we are in case (a) of the 
theorem, and that when $a\ne b$ one has  
\[
T_2(P;\mathbf h)=4P.
\]

\par We have shown that when $P$ is large enough and $T_2(P;\mathbf h)>P^\eta$, then we are either in case (a) 
or in case (b) of the theorem. In all other circumstances, we have confirmed that 
$T_2(P;\mathbf h)=O(P^{\varepsilon})$. The proof of the theorem is therefore complete.
\end{proof}

We follow the strategy of the proof of Corollary \ref{corollary1.2} in our proof of Corollary \ref{corollary1.4}.

\begin{proof}[Proof of Corollary \ref{corollary1.4}] We proceed using an argument almost identical to that 
establishing Corollary \ref{corollary1.2}. We may suppose throughout that $\mathbf h\ne \mathbf 0$. One has
\begin{equation}\label{3.10}
T_3(P;\mathbf h)=\sum_{1\le x_3,y_3\le P}T_2(P;\mathbf h'(x_3,y_3)),
\end{equation}
where $\mathbf h'=(h_1',h_2',h_4')$ is defined as in \eqref{2.14}. Let $\eta>0$ be arbitrarily small. From Theorem 
\ref{theorem1.3}, we know that $T_2(P;\mathbf h')\le P^\eta$ unless, for some positive integers $a$ and $b$ with 
$1\le a,b\le P$, one has $h_j'=a^j-b^j$ $(j=1,2,4)$, in which case
\begin{equation}\label{3.11}
x_3^j-y_3^j+a^j-b^j=h_j\quad (j=1,2,4).
\end{equation}
When $a\ne b$, one has $T_2(P;\mathbf h')=4P$, and when instead we have $a=b$, then one has 
$T_2(P;\mathbf h')=2P^2-P$. We therefore conclude from \eqref{3.10} that
\begin{equation}\label{3.12}
T_3(P;\mathbf h)\ll P^{2+\varepsilon}+P\Upsilon_1+P^2\Upsilon_2,
\end{equation}
where $\Upsilon_1$ now denotes the number of solutions of \eqref{3.11} with $1\le a,b,x_3,y_3\le P$ and 
$a\ne b$, and $\Upsilon_2$ denotes the corresponding number of solutions with $a=b=1$.\par

Again applying Theorem \ref{theorem1.3}, we find that $\Upsilon_1=O(P)$ whenever $\mathbf h\ne \mathbf 0$. 
In order to bound $\Upsilon_2$, we may assume that $a=b=1$, and observe that the equations \eqref{3.11} yield 
$h_1=h_2=h_4=0$ whenever $h_1=0$. When $h_1\ne 0$, meanwhile, we find as in the discussion concluding the 
proof of Corollary \ref{corollary1.2} that
\[
x_3=\tfrac{1}{2}(h_2/h_1+h_1)\quad \text{and}\quad y_3=\tfrac{1}{2}(h_2/h_1-h_1).
\]
Hence, when $a=b=1$, one sees that $\Upsilon_2=O(1)$. Then we conclude from \eqref{3.12} that whenever 
$\mathbf h\ne \mathbf 0$, one has $T_3(P;\mathbf h)\ll P^{2+\varepsilon}$. This completes the proof of the 
corollary.
\end{proof}

\section{Paucity in affine Br\"udern-Robert systems}
In this section we investigate the Diophantine system \eqref{1.5} previously considered, in the case 
$\mathbf h=\mathbf 0$, by Br\"udern and Robert \cite{BR2012}. In order to illustrate our ideas without the 
distraction of superficial complications, we restrict attention to the situation with $s=k+1$. Thus, when 
$\mathbf h\in \mathbb Z^k$, we consider the number of integral solutions $U_{k+1,k}(P;\mathbf h)$ of the system 
of equations
\begin{equation}\label{4.1}
\sum_{i=1}^{k+1}x_i^{2j-1}=h_j\quad (1\le j\le k),
\end{equation}
with $|x_i|\le P$ $(1\le i\le k+1)$.

Consider a fixed solution $(x_1,\ldots ,x_s)$ of the system \eqref{1.5} counted by $U_{s,k}(P;\mathbf h)$. It is possible 
that one may relabel variables in such a manner that, for some integer $j$ with $0\le j\le s/2$, one has
\begin{equation}\label{4.2}
x_{2i-1}+x_{2i}=0\quad (1\le i\le j),
\end{equation}
and so that for no $l$ and $m$ with $2j<l<m\le s$ does one have
\begin{equation}\label{4.3}
x_l+x_m=0.
\end{equation}
We shall describe an $s$-tuple $\mathbf x$ having this property as being of {\it type $j$}. We write 
$V_{s,k}^{(j)}(P;\mathbf h)$ for the number of solutions $\mathbf x$ of \eqref{1.5} counted by $U_{s,k}(P;\mathbf h)$ 
having type $j$.\par

Next, for a fixed $k$-tuple $\mathbf h$, we denote by $I_k(\mathbf h)$ the largest integer $j$ 
having the property that there exists some $(k+1)$-tuple $\mathbf x$ of type $j$, with $|x_i|\le P$ $(1\le i\le k+1)$, 
satisfying \eqref{4.1}. Thus, we have $I_k(\mathbf h)\le (k+1)/2$, and
\begin{equation}\label{4.4}
U_{k+1,k}(P;\mathbf h)=\sum_{0\le j\le I_k(\mathbf h)}V_{k+1,k}^{(j)}(P;\mathbf h).
\end{equation}
Observe that, when $V_{k+1,k}^{(j)}(P;\mathbf h)>0$, the relations \eqref{4.2} ensure that 
$V_{k+1,k}^{(j)}(P;\mathbf h)\ge P^j$. It follows, in particular, that
\begin{equation}\label{4.5}
U_{k+1,k}(P;\mathbf h)\ge P^{I_k(\mathbf h)}.
\end{equation}
Our proofs of Theorems \ref{theorem1.5} and \ref{theorem1.6} make critical use of a corresponding upper bound 
which is a consequence of the following auxiliary lemma. 

\begin{lemma}\label{lemma4.1} One has $V_{k+1,k}^{(j)}(P;\mathbf h)\ll P^{j+\varepsilon}$.
\end{lemma}

From the upper bound supplied by this lemma, we infer via \eqref{4.4} that
\begin{equation}\label{4.6}
U_{k+1,k}(P;\mathbf h)\ll \sum_{0\le j\le I_k(\mathbf h)}V_{k+1,k}^{(j)}(P;\mathbf h)\ll P^{I_k(\mathbf h)+\varepsilon}.
\end{equation}

\begin{proof}[Proof of Lemma \ref{lemma4.1}] If $j=(k+1)/2$, which is possible only when $k$ is odd, then the 
relations \eqref{4.2} combine with \eqref{4.1} to show that $h_j=0$ $(1\le j\le k)$. We relabel variables so as to 
realise the relations \eqref{4.2}. Then, given fixed choices for the variables $x_{2i-1}$ $(1\le i\le \tfrac{1}{2}(k+1))$, 
the remaining variables are determined uniquely by the relations \eqref{4.2}. Thus, in this case in which $j=(k+1)/2$, 
we have $V_{k+1,k}^{((k+1)/2)}(P;\mathbf h)\ll P^{(k+1)/2}$, and the conclusion of the lemma follows.\par

Consider now the situation with $j<(k+1)/2$. Here, by relabelling variables, we find from \eqref{4.2} and \eqref{4.3} 
that
\begin{equation}\label{4.7}
V_{k+1,k}^{(j)}(P;\mathbf h)\ll P^jV_{k+1-2j,k}^{(0)}(P;\mathbf h).
\end{equation}

The situation in which $2j\in \{k-1,k\}$ is straightforward to handle. Here, we have $k+1-2j\in \{1,2\}$. The number 
of solutions of the equation
\[
x_1=h_1
\]
is plainly $1$, and so $V_{1,k}^{(0)}(P;\mathbf h)\le V_{1,1}^{(0)}(P;h_1)=1$. Meanwhile, the number of solutions of 
the system of equations
\[
\left. \begin{aligned}x_1^3+x_2^3&=h_2\\ x_1+x_2&=h_1\end{aligned}\right\}
\]
with $x_1+x_2\ne 0$ is also $O(1)$. In order to confirm this claim, we observe that since $h_1=x_1+x_2\ne 0$, we 
have
\[
x_1^2-x_1x_2+x_2^2=h_2/h_1\quad \text{and}\quad x_1+x_2=h_1,
\]
whence
\[
h_2/h_1=x_1^2-x_1(h_1-x_1)+(h_1-x_1)^2=3x_1^2-3x_1h_1+h_1^2.
\]
This quadratic equation has at most 2 solutions for $x_1$, each of which uniquely determines $x_2=h_1-x_1$. Thus, 
in either of these cases, we have $V_{k+1-2j,k}^{(0)}(P;\mathbf h)=O(1)$, and the estimate \eqref{4.7} then 
yields the bound $V_{k+1,k}^{(j)}(P;\mathbf h)\ll P^{j+\varepsilon}$. The conclusion of the lemma again follows, 
therefore, when $2j\in \{k-1,k\}$.\par

When instead $0\le 2j\le k-2$, we proceed in a more sophisticated manner.
By discarding information from the highest degree equations in \eqref{4.1}, we are led from the relation \eqref{4.7} to the 
estimate
\begin{equation}\label{4.8}
V_{k+1,k}^{(j)}(P;\mathbf h)\ll P^jV_{k+1-2j,k-2j}^{(0)}(P;\mathbf h')\quad (0\le j\le \tfrac{1}{2}(k-2)),
\end{equation}
where we write
\begin{equation}\label{4.9}
\mathbf h'=(h_1,\ldots ,h_{k-2j}).
\end{equation}
\par

Put $\kappa=k-2j$, and examine the expression $V_{\kappa+1,\kappa}^{(0)}(P;\mathbf h')$ that occurs in the upper 
bound \eqref{4.8}. Define the polynomials $t_j=t_j(\mathbf x)$ by putting
\[
t_j(\mathbf x)=\sum_{i=1}^{\kappa -1}x_i^{2j-1}\quad (1\le j\le \kappa).
\]
Then it follows from Br\"udern and Robert \cite[Lemma 1]{BR2012} that there exists a non-zero polynomial 
$\Upsilon (\mathbf z)\in \mathbb Z[z_1,\ldots ,z_\kappa]$ having the property that
\begin{equation}\label{4.10}
\Upsilon (t_1(\mathbf x),\ldots ,t_\kappa(\mathbf x))=0.
\end{equation}
As noted in \cite[p.~227]{BR2012}, this identity also appears under alternate guise in the work of Perron 
\cite[Satz 1]{Per1955}. Furthermore, the dicussion of \cite[p.~227]{BR2012} shows that the weighted degree of 
$\Upsilon (\mathbf z)$ is equal to $\kappa (\kappa+1)/2$, in the sense that in each monomial of 
$\Upsilon (\mathbf z)$ of the shape
\[
c_{\boldsymbol \alpha}z_1^{\alpha_1}\cdots z_\kappa^{\alpha_\kappa},
\]
with $c_{\boldsymbol \alpha}\in \mathbb Z$ and $\alpha_i\ge 0$ $(1\le i\le \kappa)$, one has
\begin{equation}\label{4.11}
\sum_{i=1}^\kappa (2i-1)\alpha_i=\kappa (\kappa+1)/2.
\end{equation}

\par Next, we define the polynomials
\[
\tau_j(\mathbf x)=\sum_{i=1}^{\kappa+1}x_i^{2j-1}\quad (1\le j\le \kappa).
\]
The equations
\begin{equation}\label{4.12}
\sum_{i=1}^{\kappa+1}x_i^{2j-1}=h_j\quad (1\le j\le \kappa)
\end{equation}
now become $\tau_j(\mathbf x)=h_j$ $(1\le j\le \kappa)$. In accordance with the discussion surrounding equation (5) of \cite{BR2012}, we 
find via \eqref{4.10} that when we make the specialisation $x_\kappa+x_{\kappa+1}=0$, one has
\[
\Upsilon (\tau_1(\mathbf x),\ldots ,\tau_\kappa(\mathbf x))=\Upsilon (t_1(\mathbf x),\ldots ,t_\kappa(\mathbf x))=0,
\]
and hence the polynomial $\Upsilon(\tau_1(\mathbf x),\ldots ,\tau_\kappa(\mathbf x))$ is divisible by 
$x_\kappa+x_{\kappa+1}$. Symmetrical arguments reveal that $x_i+x_j$ is likewise a factor whenever 
$1\le i<j\le \kappa+1$. In view of the relation \eqref{4.11}, we infer that there exists a non-zero integer 
$C=C(\kappa)$ having the property that
\[
\Upsilon(\tau_1(\mathbf x),\ldots ,\tau_\kappa(\mathbf x))=C\prod_{1\le i<j\le \kappa+1}(x_i+x_j).
\]
When the variables $x_1,\ldots ,x_{\kappa+1}$ satisfy \eqref{4.12}, it therefore follows that one has
\begin{equation}\label{4.13}
C\prod_{1\le i<j\le \kappa+1}(x_i+x_j)=\Upsilon(h_1,\ldots ,h_\kappa).
\end{equation}

\par Recall from \eqref{4.9} our definition of $\mathbf h'$, and suppose that $\Upsilon (\mathbf h')=0$. Then it 
follows from \eqref{4.13} that $x_i+x_j=0$ for some indices $i$ and $j$ with $1\le i<j\le \kappa+1$. This conclusion 
is in conflict with the definition of $V_{\kappa+1,\kappa}^{(0)}(P;\mathbf h')$, since any solutions of \eqref{4.12} 
counted by this quantity are of type $0$. This contradiction shows, in fact, that it is not possible that 
$\Upsilon(\mathbf h')=0$.\par

We may therefore continue our deliberations under the assumption that $\Upsilon(\mathbf h')\ne 0$. Then since $\Upsilon (\mathbf z)$ has weighted degree 
$\kappa (\kappa+1)/2$, we see that $|\Upsilon (\mathbf h')|\ll P^{\kappa (\kappa+1)/2}$, and hence we find from 
an elementary estimate for the divisor function that there are $O(P^\varepsilon)$ possible choices for integers 
$d_{ij}$ $(1\le i<j\le \kappa+1)$ having the property that
\[
x_i+x_j=d_{ij}\quad (1\le i<j\le \kappa+1).
\]
Fix any one such choice for these integers $d_{ij}$ $(1\le i<j\le \kappa+1)$. Since for $1\le i\le \kappa+1$, one has
\[
\sum_{\substack{1\le j\le \kappa+1\\ j\ne i}}(x_i+x_j)=(\kappa-1)x_i+\sum_{j=1}^{\kappa+1}x_j=(\kappa-1)x_i+h_1,
\]
we see that
\[
(\kappa-1)x_i=\sum_{\substack{1\le j\le \kappa+1\\ j\ne i}}d_{ij}-h_1.
\]
We are at liberty to suppose that $\kappa\ge 2$. Thus, for $1\le i\le \kappa+1$, the integer $x_i$ is fixed by 
$\mathbf h'$ and our choice for $\mathbf d$, and one infers that 
$V_{\kappa+1,\kappa}^{(0)}(P;\mathbf h')\ll P^\varepsilon$.\par

We have shown that, in all circumstances, one has
\[
V_{\kappa+1,\kappa}^{(0)}(P;\mathbf h')\ll P^\varepsilon.
\]
In combination with our earlier analysis of the situation with $2j\in \{k-1,k\}$, the conclusion of the 
lemma follows upon recalling \eqref{4.8}.
\end{proof}

With Lemma \ref{lemma4.1} in hand, we are now equipped to complete the proofs of Theorems~\ref{theorem1.5} 
and \ref{theorem1.6}.

\begin{proof}[Proof of Theorem \ref{theorem1.5}]
Suppose that $\eta\in (0,1)$ and $r$ is a non-negative integer, and suppose further that 
$\mathbf h \in \mathbb Z^k$ is a coefficient tuple with the property that $U_{k+1,k}(P;\mathbf h)>P^{r+\eta}$. 
Then we see from \eqref{4.6} that
\[
P^{r+\eta}<U_{k+1,k}(P;\mathbf h)\ll P^{I_k(\mathbf h)+\varepsilon},
\]
and thus $I_k(\mathbf h)\ge r+1$. This lower bound ensures that some tuple $\mathbf x$ counted by 
$U_{k+1,k}(P;\mathbf h)$ has type at least $r+1$. Then there is a relabelling of variables with the property that
\[
x_{2i-1}+x_{2i}=0\quad (1\le i\le r+1).
\]
This is, of course, impossible when $2r+2>k+1$. By reference to \eqref{4.1}, we find that there are integers $a_i$ 
with $|a_i|\le P$ $(1\le i\le k-2r-1)$ having the property that
\[
h_j=\sum_{i=1}^{k-2r-1}a_i^{2j-1}\quad (1\le j\le k).
\]
Since $I_k(\mathbf h)\ge r+1$, moreover, it follows from \eqref{4.5} that $U_{k+1,k}(P;\mathbf h)\gg P^{r+1}$. 
Choosing $r$ to be maximal with the property that $U_{k+1,k}(P;\mathbf h)>P^{r+\eta}$ for some $\eta>0$, we 
infer that
\[
P^{r+1}\ll U_{k+1,k}(P;\mathbf h)\le P^{r+1+\eta}.
\]
Since $\eta$ may be chosen arbitrarily small, we conclude that 
$P^{r+1}\ll U_{k+1,k}(P;\mathbf h)\ll P^{r+1+\varepsilon}$. We have now confirmed all of the conclusions of 
Theorem \ref{theorem1.5}.
\end{proof}

We turn next to the problem of showing that structure in the coefficient tuple $\mathbf h$ implies precise 
paucity estimates.

\begin{proof}[Proof of Theorem \ref{theorem1.6}]
Define $r=r_0(\mathbf h)$ as in the statement of Theorem \ref{theorem1.6}, and suppose that the integers 
$a_1,\ldots ,a_r$ satisfy $|a_i|\le P$ $(1\le i\le r)$ and the equations \eqref{1.8}. Recall the definition of 
$\tau=\tau(k;\mathbf h)$ from the statement of the theorem. Then, in light of the estimates \eqref{4.5} and 
\eqref{4.6}, the proof of the theorem will follow once we confirm that $I_k(\mathbf h)=\tau(k;\mathbf h)$. Note next 
that, in view of \eqref{1.8}, any $(k+1)$-tuple $\mathbf x\in \mathbb Z^{k+1}$ of the shape 
$\mathbf x=(x_1,\ldots ,x_{k+1-r},a_1,\ldots ,a_r)$, and satisfying the equations
\[
\sum_{i=1}^{k+1-r}x_i^{2j-1}=0 \quad (1 \le j \le k),
\]
is a solution of \eqref{4.1}. It is apparent that any choice of $x_1, \ldots, x_{k+1-r}$, with $x_{2i-1}+x_{2i}=0$ for 
$1\le i\le \tau$, complemented with $x_{2\tau +1}=0$ in the case when $k+1-r$ is odd, provides such a solution. 
Thus $I_k(\mathbf h)\ge \tau$. It is also apparent that no solution of type exceeding $\tau$ can exist, since this 
would imply the existence of a representation of $\mathbf h$ as in \eqref{1.8} but with 
$r=k+1-2(\tau+1)<r_0(\mathbf h)$, contradicting the implicit minimality of $r_0(\mathbf h)$. Thus, we have 
$I_k(\mathbf h)=\tau(k;\mathbf h)$ as claimed, and the conclusion of the theorem follows. 
\end{proof}

Finally, we turn to the proof of Corollary \ref{corollary1.7}.

\begin{proof}[Proof of Corollary \ref{corollary1.7}]
Suppose that $\mathbf h\in \mathbb Z^k\setminus \{ \mathbf 0\}$ and $1\le t\le k$. We observe first that it suffices 
to establish the conclusion of the corollary in the cases with $\lceil (k+1)/2\rceil\le t\le k$ in order to establish it in 
full generality for $1\le t\le k$. In order to confirm this observation, suppose that $t$ is an integer with 
$1\le t<\lceil (k+1)/2\rceil$. By considering solutions of the system \eqref{1.5} in the case $s=2t+2$ with 
$x_{2t+1}=-x_{2t+2}$, we see that
\begin{equation}\label{4.14}
PU_{2t,k}(P;\mathbf h)\ll U_{2t+2,k}(P;\mathbf h).
\end{equation}
By iterating this relation, one finds that whenever it is known that
\begin{equation}\label{4.15}
U_{2T,k}(P;\mathbf h)\ll   P^{T-1+\varepsilon}\quad (\mathbf h\in \mathbb Z^k\setminus \{0\}),
\end{equation}
with some integer $T$ satisfying $T\ge \lceil (k+1)/2\rceil$, then it follows that
\[
U_{2t,k}(P;\mathbf h)\ll P^{t-T}\cdot P^{T-1+\varepsilon} \ll P^{t-1+\varepsilon}\quad (1\le t\le T).
\]
This confirms our claim.\par

We next set about proving the upper bound \eqref{4.15} for $\lceil (k+1)/2\rceil \le T\le k$. We observe that
\begin{equation}\label{4.16}
U_{2T,k}(P;\mathbf h)=\sum_{|x_{k+2}|\le P}\cdots \sum_{|x_{2T}|\le P}U_{k+1,k}(P;\mathbf h'),
\end{equation}
where we write
\[
h_j'=h_j-(x_{k+2}^{2j-1}+\ldots +x_{2T}^{2j-1})\quad (1\le j\le k).
\]
Let $\eta>0$ be arbitrarily small, and suppose that $\mathbf h'\in \mathbb Z^k$ is a coefficient $k$-tuple with the 
property that for some non-negative integer $r$, one has
\begin{equation}\label{4.17}
U_{k+1,k}(P;\mathbf h')>P^{r+\eta}.
\end{equation}
We take $r=r(\mathbf h')$ to be maximal with this property, and note that from Theorem \ref{theorem1.5}, one 
then has $0\le r\le (k-1)/2$. The remaining tuples $\mathbf h'$ in which \eqref{4.17} does not hold for any such 
value of $r$ must have the property that
\begin{equation}\label{4.18}
U_{k+1,k}(P;\mathbf h')\ll P^\varepsilon.
\end{equation}
Note that when \eqref{4.17} holds with $0\le r\le (k-1)/2$, it also follows from Theorem \ref{theorem1.5} that there 
exist integers $a_i$ with $|a_i|\le P$ $(1\le i\le k-1-2r)$ with
\begin{equation}\label{4.19}
h_j-(x_{k+2}^{2j-1}+\ldots +x_{2T}^{2j-1})=a_1^{2j-1}+\ldots +a_{k-1-2r}^{2j-1}\quad (1\le j\le k).
\end{equation}
In these circumstances, Theorem \ref{theorem1.5} shows that
\[
U_{k+1,k}(P;\mathbf h')\ll P^{r+1+\varepsilon}.
\]
The number of possible choices for $\mathbf h'$ of this type is determined by \eqref{4.19}, and so this number is 
bounded above by the number of integral solutions of the system
\[
\sum_{i=1}^{2T-2r-2}y_i^{2j-1}=h_j\quad (1\le j\le k),
\]
with $|y_i|\le P$ $(1\le i\le 2T-2r-2)$. We thus conclude from \eqref{4.16} and \eqref{4.18} that
\[
U_{2T,k}(P;\mathbf h)\ll P^{2T-k-1+\varepsilon}+\sum_{0\le r\le (k-1)/2}P^{r+1+\varepsilon}
U_{2T-2r-2,k}(P;\mathbf h).
\]
By employing \eqref{4.14}, we see that for $0\le r\le (k-1)/2$, one has
\[
P^{r+1+\varepsilon}U_{2T-2r-2,k}(P;\mathbf h)\ll P^{1+\varepsilon}U_{2T-2,k}(P;\mathbf h).
\]
Thus, we may conclude at this stage that when $\lceil (k+1)/2\rceil \le T\le k$, one has
\begin{equation}\label{4.20}
U_{2T,k}(P;\mathbf h)\ll P^{T-1+\varepsilon}+P^{1+\varepsilon}U_{2T-2,k}(P;\mathbf h).
\end{equation}

\par We now divide into cases according to the parity of $k$. When $k$ is odd, we may iterate the application of 
\eqref{4.20} to infer that
\begin{equation}\label{4.21}
U_{2T,k}(P;\mathbf h)\ll P^{T-1+\varepsilon}+P^{T-(k+1)/2}U_{k+1,k}(P;\mathbf h).
\end{equation}
We now see from Theorem \ref{theorem1.5} that
\[
U_{k+1,k}(P;\mathbf h)\ll P^{(k-1)/2+\varepsilon}
\]
unless $h_j=0$ $(1\le j\le k)$. Since the latter scenario is specifically excluded by the hypotheses of the corollary, it 
follows from \eqref{4.21} that
\[
U_{2T,k}(P;\mathbf h)\ll P^{T-1+\varepsilon}+P^{T-(k+1)/2}\cdot P^{(k-1)/2+\varepsilon}\ll P^{T-1+\varepsilon}.
\]
This confirms \eqref{4.15} in the scenario that $k$ is odd.\par

Consider next the situation in which $k$ is even. Here, we deduce from \eqref{4.20} that
\begin{equation}\label{4.22}
U_{2T,k}(P;\mathbf h)\ll P^{T-1+\varepsilon}+P^{T-(k+2)/2}U_{k+2,k}(P;\mathbf h).
\end{equation}
In this situation, we make use of the bound
\[
U_{k+2,k}(P;\mathbf h)=\sum_{|x|\le P}U_{k+1,k}(P;\mathbf h'),
\]
where we now put
\begin{equation}\label{4.23}
h_j'=h_j-x^{2j-1}\quad (1\le j\le k).
\end{equation}
In this instance, we see from Theorem \ref{theorem1.5} that
\[
U_{k+1,k}(P;\mathbf h')\ll P^{(k-2)/2+\varepsilon},
\]
unless there exists an integer $a$ with $|a|\le P$ satisfying the system of equations
\begin{equation}\label{4.24}
h_j'=a^{2j-1}\quad (1\le j\le k),
\end{equation}
in which case we have
\[
U_{k+1,k}(P;\mathbf h')\ll P^{k/2+\varepsilon}.
\]
On recalling that $k\ge 2$, we see from the equations \eqref{4.23} and \eqref{4.24} that $h_1=x+a$. As in previous arguments we observe that if $h_1=0$, we have 
$h_j=0$ $(1\le j\le k)$, contradicting our assumption that $\mathbf h\in \mathbb Z^k\setminus \{\mathbf 0\}$. 
Thus $h_1\ne 0$, and we have
\[
h_3/h_1=x^2-ax+a^2=(x+a)^2-3a(a+x)+3a^2=h_1^2-3ah_1+3a^2.
\]
This determines $a$ up to a possible choice of roots of a quadratic equation, and hence also $x$ via the equation 
$h_1=x+a$. Consequently, in instances in which $U_{k+1,k}(P;\mathbf h')>P^{(k-2)/2+\eta}$ for some $\eta>0$, 
there are $O(1)$ choices for the variables $a$ and $x$. Then
\[
U_{k+2,k}(P;\mathbf h)=\sum_{|x|\le P}U_{k+1,k}(P;\mathbf h')\ll P\cdot P^{(k-2)/2+\varepsilon}+
1\cdot P^{k/2+\varepsilon}\ll P^{k/2+\varepsilon}.
\]
On substituting this estimate into \eqref{4.22}, we conclude that
\[
U_{2T,k}(P;\mathbf h)\ll P^{T-1+\varepsilon}+P^{T-(k+2)/2+\varepsilon}\cdot P^{k/2+\varepsilon}\ll 
P^{T-1+2\varepsilon}.
\]
This confirms \eqref{4.15} in the scenario that $k$ is even.\par

In view of the preliminary discussion opening this argument, the proof of the corollary is now complete.
\end{proof}

\bibliographystyle{amsbracket}
\providecommand{\bysame}{\leavevmode\hbox to3em{\hrulefill}\thinspace}

\end{document}